\documentclass[11pt]{amsart}%
\usepackage{amscd,amsmath,latexsym,amsthm,amsfonts,amssymb,graphicx,color,geometry,hyperref}
\usepackage[utf8]{inputenc}
\usepackage{amsmath}
\usepackage{amsfonts}
\usepackage{amssymb}
\usepackage{graphicx}%
\providecommand{\U}[1]{\protect\rule{.1in}{.1in}}
\providecommand{\U}[1]{\protect\rule{.1in}{.1in}}

\newtheorem{theorem}{Theorem}[section]
\newtheorem{proposition}[theorem]{Proposition}
\newtheorem{corollary}[theorem]{Corollary}

\newtheorem{remark}[theorem]{Remark}

\newtheorem{lemma}[theorem]{Lemma}

\numberwithin{equation}{section}

\begin{document}
\title[Unbounded multilinear operators between Banach spaces]{On the size of the set of unbounded multilinear operators between Banach spaces}

\begin{abstract}
Among other results we investigate $\left(  \alpha,\beta\right)  $-lineability
of the set of non-continuous $m$-linear operators defined between  normed spaces as a subset of the space of all $m$-linear operators. We also give a partial answer to an open
problem on the lineability of the set of non absolutely
summing operators.

\end{abstract}
\author{V.V. F\'{a}varo, D.M. Pellegrino, P. Rueda}
\address{V.V. F\'{a}varo\\
	Faculdade de Matem\'atica\\
	Universidade Federal de Uberlandia\\
	38.400-902 - Brazil\\
	\newline email: vvfavaro@gmail.com}
\address{D. Pellegrino \\
	Departamento de Matem\'atica\\
	Universidade Federal da Para\'{\i}ba \\
	Brazil\\
    \newline email: dmpellegrino@gmail.com}
\address{P. Rueda\\
	Departamento de An\'alisis Matem\'atico\\
	Universitat de Val\`encia\\
	C/ Dr. Moliner 50, 46100 Burjassot (Valencia). Spain\\
\newline email: pilar.rueda@uv.es}

\thanks{Vin\'icius F\'{a}varo is supported by FAPEMIG Grant PPM-00217-18; and CNPq Grant 310500/2017-6.\\ D. Pellegrino is supported by CNPq and Grant 2019/0014 Paraiba State Research Foundation (FAPESQ).\\ P. Rueda  is supported by the Ministerio de Econom\'{\i}a, Industria y Competitividad and FEDER under project MTM2016-77054-C2-1-P}
\subjclass[2010]{Primary 15A03; Secondary 46G25 46E10 47H60}
\keywords{lineability, multilinear operators}
\maketitle

\section{Introduction}

Given any infinite-dimensional normed space $E$ and any normed space $F$, it
is well-known and easy to show that there exist linear operators from $E$ to
$F$ that fail to be continuous. Let us denote the set of such discontinuous
operators by $\mathcal{NL}\left(  E;F\right)  ;$ it is obvious that it is not
a vector space but what can be said about the size of this set? The word
\textquotedblleft size\textquotedblright\ is not precise and in this paper we
are interested in the linear viewpoint: is there any linear structure inside
$\mathcal{NL}\left(  E;F\right)  $? This line of investigation is what we call
\textquotedblleft lineability\textquotedblright. If $V$ is a vector space and
$\beta$ is a cardinal number, a subset $A$ of $V$ is called \emph{$\beta
\text{-lineable}$} in $V$ if $A\cup\left\{  0\right\}  $ contains a $\beta
$-dimensional linear subspace of $V$. This notion was introduced in the
classical paper \cite{Aron} by Aron, Gurariy, and Seoane-Sep\'{u}lveda (see
also \cite{book2}) and since then it has been successfully investigated in
different settings.

All along this paper $E,F$ are normed spaces over the real scalar-field
$\mathbb{R}$ and the dimension of a vector space $V$ is denoted by $\dim V;$
also $\aleph_{0}$ denotes the first infinite cardinal number and
$\mathfrak{c}$ denotes the continuum. According to \cite{da}, if $\alpha
,\beta$ are cardinal numbers and $\alpha<\beta,$ a subset $A$ of a vector
space $V$ is \emph{$\left(  \alpha,\beta\right)  $-lineable} in $V$ if $A$ is
$\alpha$-lineable in $V$ and for every subspace $W_{\alpha}\subset V$ with
$W_{\alpha}\subset A\cup\left\{  0\right\}  $ and $\dim W_{\alpha}=\alpha$,
there is a subspace $W_{\beta}\subset V$ with $\dim W_{\beta}=\beta$ and
$W_{\alpha}\subset W_{\beta}\subset A\cup\left\{  0\right\}  $. Of course,
$\left(  0,\beta\right)  $-lineable is the same of $\beta$-lineable. This
notion is quite more restrictive than the original notion of lineability and,
in general, techniques proving $\beta$-lineability are useless to prove
$\left(  \alpha,\beta\right)  $-lineability. For instance, in \cite{da} the
authors prove that%
\[
L_{p}[0,1]\setminus%
%TCIMACRO{\tbigcup \limits_{q>p}}%
%BeginExpansion
{\textstyle\bigcup\limits_{q>p}}
%EndExpansion
L_{q}[0,1]
\]
is $\left(  1,\mathfrak{c}\right)  $-lineable, and the case of $\left(
a,\mathfrak{c}\right)  $-lineability for $a>1$ remains open.

We denote by $L(^{m}E;F)$ the vector space of all $m$-linear operators from
$E\times\cdots\times E$ to $F$, $\mathcal{L}(^{m}E;F)$ is the vector subspace
of $L(^{m}E;F)$ of all continuous $m$-linear operators and $\mathcal{NL}%
\left(  ^{m}E;F\right)  $ is the set of all noncontinuous $m$-linear operators
from $E\times\cdots\times E$ to $F$.

The $2^{\dim E}$-lineability of $\mathcal{NL}\left(  ^{m}E;F\right)  $ as a
subset of $L(^{m}E;F)$ was proven in \cite{gamez} when $F=\mathbb{R}$. In the
present paper we investigate the stronger notion of $\left(  \alpha
,\beta\right)  $-lineability in this setting, also considering $F\neq
\mathbb{R}$. One of our main results asserts that $\mathcal{NL}\left(
^{m}E;\mathbb{R}\right)  $ is $\left(  n,2^{\dim E}\right)  $-lineable for all
positive integers $n$. We also investigate vector valued multilinear mappings and prove that when $\mathfrak{c}<\dim E<\dim F$,
the set $\mathcal{NL}\left(  ^{m}E;F\right)  $ is $\left(  \tau,2^{\dim
E}\right)  $-lineable for every $\tau<\dim E$. The techniques used along the
proof of the aforementioned result are also used to give a partial answer to a
problem posed in \cite{diogo} on the lineability of the set of non absolutely
summing operators. We prove that, for  an infinite-dimensional Banach space $E$ if the set  $\mathcal{L}\left(  E;\ell_{2}(\Gamma)\right)  \setminus\Pi_{1}\left(
E;\ell_{2}(\Gamma)\right)  $ is non-void, then it is $\left(  \alpha,card(\Gamma)\right)
$-lineable for all $\alpha<card(\Gamma)$, whenever $card\left(  \Gamma\right)  >\mathfrak{c}$ and  $\dim E<card\left(
\Gamma\right)  $. Here $\Pi_{1}\left(
E;\ell_{2}(\Gamma)\right)$ denotes the space of all  absolutely
summing operators from $E$ to $\ell_{2}(\Gamma)$.

\section{Lineability of the set of unbounded multilinear operators}

We begin by recalling the following result:

\begin{theorem}
\label{NL_lineable} [see (\cite{gamez})]\label{rz}Let $E$ be an
infinite-dimensional normed space. The set $\mathcal{NL}\left(  ^{m}%
E;\mathbb{R}\right)  $ is $2^{\dim E}$-lineable.
\end{theorem}

Let us recall some fundamental results of cardinal arithmetic that will be
used several times along this paper: if $2\leq\mathfrak{b}\leq2^{\mathfrak{a}%
},$ then
\[
\mathfrak{b}^{\mathfrak{a}}=2^{\mathfrak{a}}.
\]
It is also worth remembering that the dimension of any infinite-dimensional
Banach space is at least $\mathfrak{c}$ (this does not depend on the Continuum
Hypothesis - see \cite{lacey}) and a simple consequence is that for any
infinite-dimensional Banach space $E$ we have
\begin{equation}
card(E)=\dim(E),\label{cardi}%
\end{equation}
where $card(E)$ denotes the cardinality of the set $E.$ In fact, since
$\dim(E)\geq\mathfrak{c}$, an immediate calculation shows (\ref{cardi}). We
finally recall that $\mathcal{L}\left(  ^{m}E;F\right)  $ is a Banach space
whenever $E$ is a normed space and $F$ is a Banach space.

Note that if $E$ is infinite-dimensional, then
\[
card\left(  L\left(  ^{m}E;F\right)  \right)  =\dim\left(  L\left(
^{m}E;F\right)  \right)  =\left(  cardF\right)  ^{\dim E}.
\]
In fact, since any $m$-linear operator $T$ is fully defined by its evaluation
in the basis of $E$, we have%
\[
card\left(  L\left(  ^{m}E;F\right)  \right)  =\left(  cardF\right)  ^{\left(
\dim E\right)  ^{m}}=\left(  cardF\right)  ^{\dim E},
\]
and since $\dim L\left(  ^{m}E;\mathbb{R}\right)  \geq\mathfrak{c}$, we have
\[
\dim L\left(  ^{m}E;\mathbb{R}\right)  =card\left(  L\left(  ^{m}%
E;\mathbb{R}\right)  \right)
\]
and an immediate consequence is that Theorem \ref{rz} is sharp in the sense
that the dimension $2^{\dim E}$ cannot be improved, because%
\[
\dim L\left(  ^{m}E;\mathbb{R}\right)  =card\left(  L\left(  ^{m}%
E;\mathbb{R}\right)  \right)  =\left(  card\left(  \mathbb{R}\right)  \right)
^{\left(  \dim E\right)  ^{m}}=\mathfrak{c}^{\dim E}=2^{\dim E},
\]
where the last equality holds because $\dim E\geq\mathfrak{c}$ for all Banach
spaces. The following result is an immediate corollary of Theorem
\ref{NL_lineable}, but since the argument of its proof will be used later, we
prefer to do the details.

\begin{corollary}
\label{r2}Let $E$ be an infinite-dimensional normed space and $F$ be a normed
space. The set $\mathcal{NL}\left(  ^{m}E;F\right)  $ is $\left(
\max\{2^{\dim E},\dim F\}\right)  $-lineable.
\end{corollary}

\begin{proof}
Let $\{T_{\alpha}:\alpha\in\Delta\}$ be a LI subset of $\mathcal{NL}\left(
^{m}E;\mathbb{R}\right)  $ with%
\[
span\{T_{\alpha}:\alpha\in\Delta\}\subset\mathcal{NL}\left(  ^{m}%
E;\mathbb{R}\right)  \cup\{0\}
\]
and
\[
card\left(  \Delta\right)  =2^{\dim E}.
\]
Let $v\in F$ be non null. Define%
\[
\widetilde{T_{\alpha,v}}\in\mathcal{NL}\left(  ^{m}E;F\right)
\]
by%
\[
\widetilde{T_{\alpha,v}}(x)=T_{\alpha}(x)v.
\]
Note that $\{\widetilde{T_{\alpha,v}}:\alpha\in\Delta\}$ is LI and
\[
span\{\widetilde{T_{\alpha,v}}:\alpha\in\Delta\}\subset\mathcal{NL}\left(
^{m}E;F\right)  \cup\{0\}.
\]
We have shown that $\mathcal{NL}\left(  ^{m}E;F\right)  $ is $2^{\dim E}%
$-lineable. Suppose now that $\dim F>2^{\dim E};$ in this case $\max\{2^{\dim
E},\dim F\}=\dim F$ and we proceed as follows. Choose a basis $\left\{
v_{\lambda}:\lambda\in\Gamma\right\}  $ of $F$ and $T\in\mathcal{NL}\left(
^{m}E;\mathbb{R}\right)  ;$ define%
\begin{align*}
T_{\lambda}  &  :E\times\cdots\times E\rightarrow F\\
T_{\lambda}\left(  x_{1},...,x_{m}\right)   &  =T\left(  x_{1},...,x_{m}%
\right)  v_{\lambda}.
\end{align*}
It is easy to verify that $span\left\{  T_{\lambda}:\lambda\in\Gamma\right\}
$ is contained in $\mathcal{NL}\left(  ^{m}E;F\right)  \cup\{0\}$ and has
dimension $\dim F$.
\end{proof}

Recall that a subset $A$ of a vector space $V$ is called \emph{maximal
lineable} in $V$ if $A$ is $(\dim V)$-lineable.

If $\dim F\leq2^{\dim E}$, then we have
\[
\max\{2^{\dim E},\dim F\}=2^{\dim E}=\left(  \dim F\right)  ^{\dim E}=\dim
L\left(  ^{m}E;F\right)
\]
and thus we have the following consequence:

\begin{corollary}
If $E,F$ are infinite dimensional Banach spaces, with $\dim F\leq2^{\dim E}$,
then the set $\mathcal{NL}\left(  ^{m}E;F\right)  $ is maximal lineable.
\end{corollary}

Our main goal is to prove stronger lineability properties of $\mathcal{NL}%
\left(  ^{m}E;F\right)  $. If $E$ is a normed space we recall that there is a
normed space $\otimes_{\pi}^{m}E$ (the tensor product with the projective
tensor norm) such that $L\left(  ^{m}E;F\right)  $ is isomorphic to $L\left(
\otimes_{\pi}^{m}E;F\right)  $ and preserves continuity. Let us denote this
isomorphism by%
\[
\Psi:L\left(  \otimes_{\pi}^{m}E;F\right)  \rightarrow L\left(  ^{m}%
E;F\right)
\]
and the restriction of $\Psi$ to $\mathcal{L}\left(  \otimes_{\pi}%
^{m}E;F\right)  $, still denoted by $\Psi$, is an isometric isomorphism
\[
\Psi:\mathcal{L}\left(  \otimes_{\pi}^{m}E;F\right)  \rightarrow
\mathcal{L}\left(  ^{m}E;F\right)  .
\]
We refer the interested reader to the monograph \cite{ryan} for details on
tensor products of Banach spaces. Now we prove some simple \textit{lemmata} to
prove our first main theorem.

\begin{lemma}
\label{787}$\mathcal{NL}\left(  ^{m}E;F\right)  $ is $\left(  \tau
,\alpha\right)  $-lineable in $L\left(  ^{m}E;F\right)  $ if and only if
$\mathcal{NL}\left(  \otimes_{\pi}^{m}E;F\right)  $ is $\left(  \tau
,\alpha\right)  $-lineable in $L\left(  \otimes_{\pi}^{m}E;F\right)  $.
\end{lemma}

\begin{proof}
($\Leftarrow$) Let $W_{0}$ be a subspace of $L\left(  ^{m}E;F\right)  $
contained in $\mathcal{NL}\left(  ^{m}E;F\right)  \cup\{0\}$ with $\dim
W_{0}=\tau.$ Thus $W_{1}:=\Psi^{-1}(W_{0})\subset\mathcal{NL}\left(
\otimes_{\pi}^{m}E;F\right)  \cup\{0\}$ is such that $\dim W_{1}=\tau$. \ By
hypothesis there is a vector space $V$ with $\dim\left(  V\right)  =\alpha,$
such that%
\[
W_{1}\subset V\subset\mathcal{NL}\left(  \otimes_{\pi}^{m}E;F\right)
\cup\{0\}.
\]
Then%
\[
W_{0}=\Psi(W_{1})\subset\Psi(V)\subset\Psi\left(  \mathcal{NL}\left(
\otimes_{\pi}^{m}E;F\right)  \cup\{0\}\right)  =\mathcal{NL}\left(
^{m}E;F\right)  \cup\{0\}.
\]
Since $\dim\Psi(V)=\dim V=\alpha$, the proof is done. The converse is similar.
\end{proof}

\begin{lemma}
\label{090}If $\dim E=\alpha\geq\aleph_{0}$, then $\dim\left(  \otimes_{\pi
}^{m}E\right)  =\alpha$.\ 
\end{lemma}

\begin{proof}
In fact if $\{b_{i}:i\in\Delta\}$ is a basis of $E$, then the set
$\{(b_{i_{1}}\otimes b_{i_{2}}\otimes\cdots\otimes b_{i_{m}}):i_{1}%
,..,i_{m}\in\Delta\}$ is a basis of $E\otimes\cdots\otimes E$. So
\[
\dim\left(  \otimes_{\pi}^{m}E\right)  =\alpha^{m}=\alpha.
\]

\end{proof}

Now we state and prove our first main theorem:

\begin{theorem}
\label{dimKer} \label{0101}Let $E$ be an infinite-dimensional normed space and
$F$ be a normed space. Let $n\in\mathbb{N}$. Consider linearly independent
linear operators $T_{1},\ldots,T_{n} \in\mathcal{NL}\left(  E;F\right)  $ so
that $span (T_{1},\ldots, T_{n})$ is contained in $\mathcal{NL}\left(
E;F\right)  \cup\{ 0\}$. If $\cap_{i=1}^{n} Ker T_{i}$ is non trivial and
infinite dimensional (with dimension $\gamma$) then, there exists a subspace
$S$ of $L(E;F)$ such that
\[
span(T_{1},\ldots,T_{n})\subset S\subset\mathcal{NL}\left(  E;F\right)
\cup\{0\}
\]
and $\dim S\geq\max(2^{\gamma},\dim F)$.
\end{theorem}

\begin{proof}
Let us call $K:=\cap_{i=1}^{n}KerT_{i}$ and let $\delta:=\max\left\{
2^{\gamma},\dim F\right\}  $. By Corollary \ref{r2}, $\mathcal{NL}\left(
K;F\right)  $ is $\delta$-lineable. Let $S_{\eta}:K\rightarrow F$ be LI with
$card(\{S_{\eta}:\eta\in\Phi\})=\delta$ and $span(\{S_{\eta}:\eta\in
\Phi\})\subset\mathcal{NL}\left(  K;F\right)  \cup\{0\}$. Let $\mathcal{B}$ be
a normalized Hamel basis of $K$ and complete it to a basis $\mathcal{E}$ of
$E$. Define
\[
\widetilde{S_{\eta}}(x)=\left\{
\begin{array}
[c]{c}%
S_{\eta}(x)\ \text{ if }x\in\mathcal{B}\\
0\ \text{ if }x\in\mathcal{E}-\mathcal{B}%
\end{array}
\right.  \text{ }%
\]
Note that the set $\left\{  \widetilde{S_{\eta}}:\eta\in\Phi\right\}  $ is LI
and
\[
card\left\{  \widetilde{S_{\eta}}:\eta\in\Phi\right\}  =\delta.
\]
Also,
\[
span\left\{  \widetilde{S_{\eta}}:\eta\in\Phi\right\}  \subset\mathcal{NL}%
\left(  E;F\right)  \cup\{0\}.
\]
Note also that we can add $T_{1},\ldots,T_{n}$ to the set $\left\{
\widetilde{S_{\eta}}:\eta\in\Phi\right\}  $ without loosing our properties.
First note that%
\[
\left\{  T_{1},\ldots,T_{n},\widetilde{S_{\eta}}:\eta\in\Phi\right\}
\]
is LI. In fact, if
\[
a_{1}T_{1}+\cdots+a_{n}T_{n}+b_{1}\widetilde{S_{\eta_{1}}}+\cdots
+b_{k}\widetilde{S_{\eta_{k}}}=0,
\]
then, in particular, for all $x\in\mathcal{E}-\mathcal{B}$, we have%
\[
a_{1}T_{1}(x)+\cdots+a_{n}T_{n}(x)=0
\]
and thus $a_{1}T_{1}+\cdots+a_{n}T_{n}=0$. Therefore
\[
a_{1}=\cdots=a_{n}=0.
\]
As a consequence,
\[
b_{1}\widetilde{S_{\eta_{1}}}+\cdots+b_{k}\widetilde{S_{\eta_{k}}}=0,
\]
and so, restricting to $x\in K$ we get
\[
b_{1}S_{\eta_{1}}(x)+\cdots+b_{k}S_{\eta_{k}}(x)=0
\]
and the linear independence of $S_{\eta}$ implies that
\[
b_{1}=\cdots=b_{k}=0.
\]
Moreover, $span\left\{  T_{1},\ldots,T_{n},\widetilde{S_{\eta}}:\eta\in
\Phi\right\}  \subset\mathcal{NL}\left(  E;F\right)  \cup\{0\}.$ Indeed, for
any linear combination we have
\begin{align*}
\left\Vert a_{1}T_{1}+\cdots+a_{n}T_{n}+\sum_{j=1}^{n}b_{j}\widetilde
{S_{\eta_{j}}}\right\Vert  &  \geq\sup_{x\in B_{K}}\left\Vert a_{1}%
T_{1}(x)+\cdots+a_{n}T_{n}(x)+\sum_{j=1}^{n}b_{j}\widetilde{S_{\eta_{j}}%
}(x)\right\Vert \\
&  =\sup_{x\in B_{K}}\left\Vert \sum_{j=1}^{n}b_{j}\widetilde{S_{\eta_{j}}%
}(x)\right\Vert \\
&  =\infty.
\end{align*}

\end{proof}

\begin{corollary}
\label{uno} Let $E$ be an infinite dimensional normed space and let $F$ be a
normed space with $\dim E> \dim F$. For any $m\in\mathbb{N}$, the set
$\mathcal{NL}\left( ^{m} E;F\right) $ is $(1, 2^{\dim E})$-lineable and
$2^{\dim E}$ cannot be replaced with a bigger cardinal number.
\end{corollary}

\begin{proof}
We prove the linear case $m=1$. The multilinear case follows then from Lemma
\ref{787} and Lemma \ref{090}. Let $T\in\mathcal{NL}\left(  E;F\right)  $ be
non null. Let $\lambda:=\dim E$ and $\beta:=\dim F$. Since 
\[
\dim KerT+\dim ImT=\lambda,
\]
we have
\[
\lambda\leq\dim KerT+\beta.
\]
Since $\lambda>\beta,$ we have
\[
\dim KerT=\lambda.
\]
By Theorem \ref{dimKer}, there is a subspace $S$ of $L(E;F)$ such that
\[
T\in S\subset\mathcal{NL}\left(  E;F\right)  \cup\{0\}
\]
and $\dim S=\max(2^{\lambda},\beta)=2^{\lambda}$. Besides, $2^{\dim E}$ is
maximal as $2^{\dim E}=\left(  cardF\right)  ^{\dim E}=\dim L\left(
^{m}E;F\right)  $.
\end{proof}

In particular, $\mathcal{NL}\left(  ^{m}E;\mathbb{R}\right)  $ is $\left(
1,2^{\dim E}\right)  $-lineable. The next corollary shows that we can go
further in the scalar case.

\medskip
%%%%%%%%%%%%%%%%%%%%%%%%%%%%%%%%%%%%%%%%%%%%%%%%%%%%

\begin{corollary}
Let $E$ be an infinite dimensional normed space and let $m\in\mathbb{N}$. The
set $\mathcal{NL}\left( ^{m} E;\mathbb{R}\right) $ is $(n, 2^{\dim E}%
)$-lineable for any $n\in\mathbb{N}$, and $2^{\dim E}$ cannot be replaced with
a bigger cardinal number.
\end{corollary}

\begin{proof}
Once more, it is enough to prove the case $m=1$, as the multilinear case
follows from Lemma \ref{787} and Lemma \ref{090}. \newline\indent
Let $n\in\mathbb{N}$ and let $T_{1}, \ldots, T_{n}\in\mathcal{NL}\left(
E;\mathbb{R}\right) $ be linearly independent so that $span(T_{1},\ldots
,T_{n})\subset\mathcal{NL}\left(  E;\mathbb{R}\right) \cup\{0\}$. Let
$\lambda:=\dim E$. By Theorem \ref{dimKer} it is sufficient to prove that
$\cap_{i=1}^{n} Ker T_{i}$ is infinite dimensional. But this follows from the
second isomorphism theorem for vector spaces, that is, the quotient spaces
$(Ker T_{1} +Ker T_{2})/(Ker T_{2})$ and $(Ker T_{1})/(Ker T_{1}\cap Ker
T_{2})$ are isomorphic. As $Ker T_{1} +Ker T_{2}=E$, $Ker T_{2}$ has
co-dimension $1$ and $\dim Ker T_{1}=\lambda$, it follows that $\dim(Ker
T_{1}\cap Ker T_{2})=\lambda$. Repeating the argument finitely many times, we
get that $\dim(\cap_{i=1}^{n} Ker T_{i})=\lambda$. Theorem \ref{dimKer} tells
us that there exists a subspace $S$ of $L(E;\mathbb{R})$ such that
\[
span(T_{1},\ldots,T_{n})\subset S\subset\mathcal{NL}\left(  E;\mathbb{R}%
\right) \cup\{0\}
\]
and $\dim S\geq2^{\lambda}$. \newline\indent
Besides, $2^{\dim E}$ is maximal as $2^{\dim E}=\dim L\left(  ^{m}%
E;\mathbb{R}\right) $.
\end{proof}

Our second main result (Theorem \ref{333}) shows   that we can also get   lineability properties whenever $F$ has a bigger dimension than $E$. First we need to prove a lemma for the case of linear operators. Given
an infinite cardinal number $\tau_{0},$ consider
\[
\mathcal{NL}_{\tau_{0}}\left(  E;F\right)  :=\left\{  T\in\mathcal{NL}\left(
E;F\right)  :\dim(T(E))\leq\tau_{0}<\dim F\right\}  .
\]
The next lemma shows, under additional hypotheses on $E,F$, that the set
$\mathcal{NL}_{\tau_{0}}\left(  ^{m}E;F\right)  $ is $\left(  \tau,2^{\dim
E}\right)  $-lineable for all $\tau<\dim F$; this may sound strange because
$\mathcal{NL}_{\tau_{0}}\left(  ^{m}E;F\right)  \subset\mathcal{NL}\left(
^{m}E;F\right)  $, but it is a natural phenomenon when dealing with this
stronger notion of lineability.

\begin{lemma}
\label{xc}Let $E,F$ be normed spaces with $\dim E>\mathfrak{c}$ and $\dim
F>\mathfrak{c}$ and consider an infinite cardinal number $\tau_{0}<\dim F$.
The set $\mathcal{NL}_{\tau_{0}}\left(  E;F\right)  $ is $\left(  \tau,2^{\dim
E}\right)  $-lineable for $\tau<\min\{\dim E,\dim F\}$.
\end{lemma}

\begin{proof}
Note that the construction of Corollary \ref{r2} shows that $\mathcal{NL}%
_{\tau_{0}}\left(  E;F\right)  $ is $\tau$-lineable.

Let $\lambda:=\dim E$ and $\beta:=\dim F$. Let $\{T_{\zeta}:\zeta\in\Omega\}$
be LI with
\[
card\left(  \Omega\right)  =\tau<\min\{\dim E,\dim F\}
\]
and
\[
span\{T_{\zeta}:\zeta\in\Omega\}\subset\mathcal{NL}_{\tau_{0}}\left(
E;F\right)  \cup\{0\}.
\]
For all $\left(  b_{j}\right)  _{j=1}^{\infty}\in\mathbb{R}^{\mathbb{N}}$, we
have%
\begin{align*}
&  card\left(
%TCIMACRO{\tbigcup \limits_{n\in\mathbb{N}}}%
%BeginExpansion
{\textstyle\bigcup\limits_{n\in\mathbb{N}}}
%EndExpansion
\text{ }%
%TCIMACRO{\tbigcup \limits_{\zeta_{1},...,\zeta_{n}\in\Omega}}%
%BeginExpansion
{\textstyle\bigcup\limits_{\zeta_{1},...,\zeta_{n}\in\Omega}}
%EndExpansion
\text{ }%
%TCIMACRO{\tbigcup \limits_{b_{1},...,b_{n}\in\mathbb{R}}}%
%BeginExpansion
{\textstyle\bigcup\limits_{b_{1},...,b_{n}\in\mathbb{R}}}
%EndExpansion
\operatorname{Im}\left(
%TCIMACRO{\tsum \limits_{j=1}^{n}}%
%BeginExpansion
{\textstyle\sum\limits_{j=1}^{n}}
%EndExpansion
b_{j}T_{\zeta_{j}}\right)  \right)  \\
&  \leq\tau_{0}\cdot\tau\cdot\mathfrak{c}\\
&  <\dim F\\
&  =card\left(  F\right)  .
\end{align*}
Consider%
\begin{equation}
v\in F-%
%TCIMACRO{\tbigcup \limits_{n\in\mathbb{N}}}%
%BeginExpansion
{\textstyle\bigcup\limits_{n\in\mathbb{N}}}
%EndExpansion
\text{ }%
%TCIMACRO{\tbigcup \limits_{\zeta_{1},...,\zeta_{n}\in\Omega}}%
%BeginExpansion
{\textstyle\bigcup\limits_{\zeta_{1},...,\zeta_{n}\in\Omega}}
%EndExpansion
\text{ }%
%TCIMACRO{\tbigcup \limits_{b_{1},...,b_{n}\in\mathbb{K}}}%
%BeginExpansion
{\textstyle\bigcup\limits_{b_{1},...,b_{n}\in\mathbb{K}}}
%EndExpansion
\operatorname{Im}\left(
%TCIMACRO{\tsum \limits_{j=1}^{n}}%
%BeginExpansion
{\textstyle\sum\limits_{j=1}^{n}}
%EndExpansion
b_{j}T_{\zeta_{j}}\right)  .\label{qwt}%
\end{equation}
For each positive integer $n$, $\left(  b_{1},...,b_{n}\right)  \neq\left(
0,...,0\right)  $ and $\zeta_{1},...,\zeta_{n}\in\Omega$, choose normalized
vectors
\[
v_{1}^{(b_{1}...b_{n})\left(  \zeta_{1},...,\zeta_{n}\right)  },v_{2}%
^{(b_{1}...b_{n})\left(  \zeta_{1},...,\zeta_{n}\right)  },...
\]
such that
\[
\lim_{i\rightarrow\infty}\left\Vert \left(
%TCIMACRO{\tsum \limits_{j=1}^{n}}%
%BeginExpansion
{\textstyle\sum\limits_{j=1}^{n}}
%EndExpansion
b_{j}T_{\zeta_{j}}\right)  v_{i}^{(b_{1}...b_{n})\left(  \zeta_{1}%
,...,\zeta_{n}\right)  }\right\Vert =\infty.
\]
Let%
\[
S=span\{v_{i}^{(b_{1}...b_{n})\left(  \zeta_{1},...,\zeta_{n}\right)  }%
:i\in\mathbb{N}\text{ and }\left(  b_{1},...,b_{n}\right)  \in\mathbb{R}%
^{n}\setminus\{(0,...,0)\}\text{ and }\zeta_{1},...,\zeta_{n}\in\Omega\}.
\]
Note that%
\[
\dim S\leq card\left(  \Omega\right)  \cdot\mathfrak{c}<\min\{\dim E,\dim
F\}\leq\dim\left(  E\right)  .
\]
Let%
\[
\{w_{i}:i\in\Lambda\}
\]
be LI, normalized, such that%
\[
\left(  span\{w_{i}:i\in\Lambda\}\right)  \oplus S=E.
\]
Let%
\[
V:=span\{w_{i}:i\in\Lambda\}.
\]
Note that 
\[
\left\{
\begin{array}
[c]{c}%
\dim\left(  E\right)  =\dim V+\dim S\\
\dim S<\dim E,
\end{array}
\right.
\]
and thus%
\[
\dim V=\dim\left(  E\right)  .
\]
By (the proof of) Corollary \ref{r2} there is
\[
\{u_{\eta}:V\rightarrow F:\eta\in\Pi\}
\]
LI, with
\[
\bigcup\limits_{\eta\in\Pi}\operatorname{Im}u_{\eta}\subset span\{v\}
\]
and
\[
card\left(  \Pi\right)  =2^{\lambda}%
\]
such that
\[
span\left(  \{u_{\eta}:V\rightarrow F:\eta\in\Pi\}\right)  \subset
\mathcal{NL}_{\tau_{0}}\left(  V;F\right)  \cup\{0\}.
\]
Let $\mathcal{A}$ be a basis of $S$ and define $\widetilde{u_{\eta}%
}:E\rightarrow F$ by%
\[
\widetilde{u_{\eta}}(x)=\left\{
\begin{array}
[c]{c}%
u_{\eta}(x)\text{ if }x\in V\\
0\text{, if }x\in\mathcal{A}%
\end{array}
.\right.  \text{ }%
\]
Note that the set $\left\{  \widetilde{u_{\eta}}:\eta\in\Pi\right\}  $ is LI
and
\[
card\left\{  \widetilde{u_{\eta}}:\eta\in\Pi\right\}  =2^{\lambda}.
\]
Also,
\[
span\left\{  \widetilde{u_{\eta}}:\eta\in\Pi\right\}  \subset\mathcal{NL}%
_{\tau_{0}}\left(  E;F\right)  \cup\{0\}.
\]
Note also that we can include $\{T_{\zeta}:\zeta\in\Omega\}$ to this vector
space without loosing our properties. In fact, first note that%
\[
\left\{  T_{\zeta},\widetilde{u_{\eta}}:\eta\in\Pi\text{ and }\zeta\in
\Omega\right\}
\]
is LI. In fact, if
\[
a_{0}^{(1)}T_{1}+...+a_{0}^{(n)}T_{n}+a_{1}\widetilde{u_{\alpha_{1}}}%
+\cdots+a_{k}\widetilde{u_{\alpha_{k}}}=0,
\]
since, for all $x\in E$,
\[
a_{0}^{(1)}T_{1}(x)+...+a_{0}^{(n)}T_{n}(x)\in span\left\{
%TCIMACRO{\tbigcup \limits_{n\in\mathbb{N}}}%
%BeginExpansion
{\textstyle\bigcup\limits_{n\in\mathbb{N}}}
%EndExpansion
\text{ }%
%TCIMACRO{\tbigcup \limits_{\zeta_{1},...,\zeta_{n}\in\Omega}}%
%BeginExpansion
{\textstyle\bigcup\limits_{\zeta_{1},...,\zeta_{n}\in\Omega}}
%EndExpansion
\text{ }%
%TCIMACRO{\tbigcup \limits_{b_{1},...,b_{n}\in\mathbb{K}}}%
%BeginExpansion
{\textstyle\bigcup\limits_{b_{1},...,b_{n}\in\mathbb{K}}}
%EndExpansion
\operatorname{Im}\left(
%TCIMACRO{\tsum \limits_{j=1}^{n}}%
%BeginExpansion
{\textstyle\sum\limits_{j=1}^{n}}
%EndExpansion
b_{j}T_{\zeta_{j}}\right)  \right\}
\]
and%
\[
a_{1}\widetilde{u_{\alpha_{1}}}(x)+\cdots+a_{k}\widetilde{u_{\alpha_{k}}%
}(x)\in span\{v\},
\]
then by (\ref{qwt}) we have%
\[
\left\{
\begin{array}
[c]{c}%
a_{0}^{(1)}T_{\zeta_{1}}(x)+...+a_{0}^{(n)}T_{\zeta_{n}}(x)=0\\
a_{1}\widetilde{u_{\alpha_{1}}}(x)+\cdots+a_{k}\widetilde{u_{\alpha_{k}}}(x)=0
\end{array}
\right.
\]
for all $x\in E.$ Hence%
\[
a_{0}^{(1)}=\cdots=a_{0}^{(n)}=a_{1}=\cdots=a_{k}=0.
\]
Now note that%
\[
span\left\{  T_{\zeta},\widetilde{u_{\eta}}:\eta\in\Pi\text{ and }\zeta
\in\Omega\right\}  \subset\mathcal{NL}_{\tau_{0}}\left(  E;F\right)
\cup\{0\}.
\]
In fact, a typical element of $span\left\{  T_{\zeta},\widetilde{u_{\eta}%
}:\eta\in\Pi\text{ and }\zeta\in\Omega\right\}  $ is written as one of the
following ways:
\[
\left\vert
\begin{array}
[c]{l}%
R_{1}=a_{0}^{(1)}T_{\zeta_{1}}+\cdots+a_{0}^{(n)}T_{\zeta_{n}}+a_{1}%
\widetilde{u_{\alpha_{1}}}+\cdots+a_{k}\widetilde{u_{\alpha_{k}}},\\
R_{2}=a_{1}\widetilde{u_{\alpha_{1}}}+\cdots+a_{k}\widetilde{u_{\alpha_{k}}%
},\\
R_{3}=a_{0}^{(1)}T_{\zeta_{1}}+\cdots+a_{0}^{(n)}T_{\zeta_{n}},
\end{array}
\right.
\]
with $a_{0}^{(1)},...,a_{0}^{(n)},a_{1},...,a_{k}\neq0.$

For $R_{1}$ note that for $x\in\mathcal{A}$ we have
\begin{align*}
\left\Vert R_{1}(x)\right\Vert  &  =\left\Vert a_{0}^{(1)}T_{\zeta_{1}%
}(x)+\cdots+a_{0}^{(n)}T_{\zeta_{n}}(x)+a_{1}\widetilde{u_{\alpha_{1}}%
}(x)+\cdots+a_{k}\widetilde{u_{\alpha_{k}}}(x)\right\Vert \\
&  =\left\Vert a_{0}^{(1)}T_{\zeta_{1}}(x)+\cdots+a_{0}^{(n)}T_{\zeta_{n}%
}(x)\right\Vert
\end{align*}
and
\[
\sup_{0\neq x\in S}\left\Vert R_{1}\left(  \frac{x}{\left\Vert x\right\Vert
}\right)  \right\Vert =\sup_{0\neq x\in S}\left\Vert a_{0}^{(1)}T_{\zeta_{1}%
}\left(  \frac{x}{\left\Vert x\right\Vert }\right)  +\cdots+a_{0}%
^{(n)}T_{\zeta_{n}}\left(  \frac{x}{\left\Vert x\right\Vert }\right)
\right\Vert =\infty.
\]
For $R_{2}$ note that
\[
\sup_{0\neq x\in V}\left\Vert a_{1}\widetilde{u_{\alpha_{1}}}\left(  \frac
{x}{\left\Vert x\right\Vert }\right)  +\cdots+a_{k}\widetilde{u_{\alpha_{k}}%
}\left(  \frac{x}{\left\Vert x\right\Vert }\right)  \right\Vert =\sup_{0\neq
x\in V}\left\Vert a_{1}u_{\alpha_{1}}\left(  \frac{x}{\left\Vert x\right\Vert
}\right)  +\cdots+a_{k}u_{\alpha_{k}}\left(  \frac{x}{\left\Vert x\right\Vert
}\right)  \right\Vert =\infty.
\]
For $R_{3}$ the argument is the same of $R_{1}$.
\end{proof}

Using the previous lemma we can prove:

\begin{theorem}
\label{333}Let $E,F$ be normed spaces and
\[
\mathfrak{c}<\dim E<\dim F.
\]
The set $\mathcal{NL}\left(  ^{m}E;F\right)  $ is $\left(  \gamma,2^{\dim
E}\right)  $-lineable for every $\gamma<\dim E$.
\end{theorem}

\begin{proof}
By Lemma \ref{787} we know that the statement is equivalent to prove that
$\mathcal{NL}\left(  \otimes_{\pi}^{m}E;F\right)  $ is $\left(  \tau,2^{\dim
E}\right)  $-lineable for every $\tau<\dim E.$ Since
\[
\dim\left(  \otimes_{\pi}^{m}E\right)  =\dim E<\dim F,
\]
and since $\dim T(E)\leq\dim E,$ using the notation of the previous lemma with
$\tau_{0}=\dim\otimes_{\pi}^{m}E$, we have%
\[
\mathcal{NL}\left(  \otimes_{\pi}^{m}E;F\right)  =\mathcal{NL}_{\tau_{0}%
}\left(  \otimes_{\pi}^{m}E;F\right)  .
\]
So, in this case, the previous lemma tells us that $\mathcal{NL}\left(
\otimes_{\pi}^{m}E;F\right)  $ is $\left(  \tau,2^{\dim E}\right)  $-lineable
for
\[
\tau<\min\{\dim\otimes_{\pi}^{m}E,\dim F\}.
\]
Since
\[
\min\{\dim\otimes_{\pi}^{m}E,\dim F\}=\dim E,
\]
we conclude that $\mathcal{NL}\left(  \otimes_{\pi}^{m}E;F\right)  $ is
$\left(  \tau,2^{\dim E}\right)  $-lineable for $\tau<\dim E$, and the proof
is done.
\end{proof}

\begin{remark}
If $\dim F>2^{\dim E}$ we can use the second part of the proof of Corollary
\ref{r2} and adapt the proof of Lemma \ref{xc} to prove that $\mathcal{NL}%
_{\tau_{0}}\left(  E;F\right)  $ is $\left(  \tau,\dim F\right)  $-lineable;
and using this result it is possible to adapt the proof of the previous
theorem to show that $\mathcal{NL}\left(  ^{m}E;F\right)  $ is $\left(
\gamma,\dim F\right)  $-lineable for every $\gamma<\dim E$.
\end{remark}

\section{Lineability \& Grothendieck's inequality}

Let $\mathbb{K}=\mathbb{R}$ or $\mathbb{C}$ and $E$ and $F$ be Banach spaces
over $\mathbb{K}.$ Recall that a continuous linear operator $T:E\rightarrow F$
is \emph{absolutely summing} when $\left(  T\left(  x_{j}\right)  \right)
_{j=1}^{\infty}\in\ell_{1}(F)$ whenever $%
%TCIMACRO{\tsum \limits_{j}}%
%BeginExpansion
{\textstyle\sum\limits_{j}}
%EndExpansion
\left\vert \varphi(x_{j})\right\vert <\infty$ for all continuous linear
functionals $\varphi:E\rightarrow\mathbb{K}$.  We denote the space of all absolutely summing
operators from $E$ to $F$ by $\Pi_{1}\left(  E;F\right)  $.
Grothendieck's inequality tells
us that every continuous linear operator $T:\ell_{1}\rightarrow H$ is
absolutely summing whenever $H$ is an infinite-dimensional Hilbert space. In
addition, Lindenstrauss and Pelczynski \cite{lin} proved that if $E$ is an
infinite dimensional Banach space with unconditional Schauder basis and every
continuous linear operator $T:E\rightarrow H$ is absolutely summing, then
$E=\ell_{1}.$ These results motivated the study whether  the set $\mathcal{L}\left(  E;F\right)
\setminus\Pi_{1}\left(  E;F\right)$ is lineable. In \cite{diogo} it was proved under some conditions on $E$ and $F$ related to the existence of unconditional basis, that $\mathcal{L}\left(  E;F\right)
\setminus\Pi_{1}\left(  E;F\right)$  is $\aleph_0$-lineable, where $\aleph_0$ is the cardinality of $\mathbb N$. This result was improved in \cite{kk}, where the conditions were weakened. Former results  on the lineability of  the set of bounded linear and non-absolutely summing operators  in certain situations can be found in \cite{PuSe}.

 In this section we show that if $\dim H>\dim E$ then the set of
continuous non-absolutely summing operators from $E$ to $H$ is void or
$\left(  \alpha,card(\Gamma)\right)  $-lineable for all $\alpha<card(\Gamma)$,
where $H=\ell_{2}(\Gamma)$. In particular, we provide a partial answer to
Problem 2.3 of \cite{diogo} (see also \cite{kk}): Under what circumstances is  $\mathcal{L}\left(  E;F\right)
\setminus\Pi_{1}\left(  E;F\right)$ $\mu$-lineable for $\mu>\aleph_0$?

\begin{lemma}
Let $E$ be an infinite-dimensional Banach space and $H$ be an
infinite-dimensional Hilbert space. If $\mathcal{L}\left(  E;H\right)
\setminus\Pi_{1}\left(  E;H\right)  $ is non-void, then it is $card(\Gamma
)$-lineable, where $H=\ell_{2}(\Gamma)$.
\end{lemma}

\begin{proof}
We split $\Gamma=%
%TCIMACRO{\tbigcup \limits_{j\in\Gamma}}%
%BeginExpansion
{\textstyle\bigcup\limits_{j\in\Gamma}}
%EndExpansion
\Gamma_{j}$ as a pairwise disjoint union with $card(\Gamma_{j})=card(\Gamma)$
for all $j$. Since $\ell_{2}\left(  \Gamma_{j}\right)  $ is isometrically
isomorphic to $\ell_{2}(\Gamma),$ we can find $T_{j}\in\mathcal{L}(E;\ell
_{2}(\Gamma_{j}))\setminus\Pi_{1}(E;\ell_{2}(\Gamma_{j})).$ Composing $T_{j}$
with the canonical inclusion $\ell_{2}(\Gamma_{j})\hookrightarrow\ell
_{2}(\Gamma)$ we obtain an operator (that we still denote by $T_{j}$)
satisfying $T_{j}\in\mathcal{L}(E;\ell_{2}(\Gamma))\setminus\Pi_{1}(E;\ell
_{2}(\Gamma)).$ Recalling that the sets $\Gamma_{j}$ are pairwise disjoint, a
simple calculation shows that $span\{T_{j}:j\in\Gamma\}\subset\mathcal{L}%
(E;\ell_{2}(\Gamma))\setminus\Pi_{1}(E;\ell_{2}(\Gamma)).$
\end{proof}

\begin{proposition}
Let $E$ be an infinite-dimensional Banach space and $H=\ell_{2}(\Gamma)$ and
$card\left(  \Gamma\right)  >\mathfrak{c}$, with $\dim E<card\left(
\Gamma\right)  $. If $\mathcal{L}\left(  E;H\right)  \setminus\Pi_{1}\left(
E;H\right)  $ is non-void, then it is $\left(  \alpha,card(\Gamma)\right)
$-lineable for all $\alpha<card(\Gamma)$.
\end{proposition}

\begin{proof}
By the previous lemma we know that $\mathcal{L}\left(  E;H\right)
\setminus\Pi_{1}\left(  E;H\right)  $ is $card(\Gamma)$-lineable. Let $V$ be a
subspace of dimension $\alpha$ such that
\[
V\setminus\{0\}\subset\mathcal{L}\left(  E;H\right)  \setminus\Pi_{1}\left(
E;H\right)  .
\]
Note that the number of coordinates of $\ell_{2}(\Gamma)$ occupied by
$\left\{  T(x):T\in V\text{ and }x\in E\right\}  $ is not bigger than
\[
\beta:=card\left(  E\right)  \cdot\aleph_{0}\cdot card\left(  V\right)
<card(\Gamma).
\]
Let us denote by $\Lambda$ the set of all such coordinates. We can split%
\[
\Gamma=\left(
%TCIMACRO{\tbigcup \limits_{j\in\Gamma}}%
%BeginExpansion
{\textstyle\bigcup\limits_{j\in\Gamma}}
%EndExpansion
\Gamma_{j}\right)
%TCIMACRO{\tbigcup }%
%BeginExpansion
{\textstyle\bigcup}
%EndExpansion
\Lambda,
\]
with $card(\Gamma_{j})=card(\Gamma)$ for all $j$ and we can repeat the
arguments of the previous lemma to obtain a subspace $W$ with $\dim
(W)=card(\Gamma)$ and
\[
V\subset W\subset\left(  \mathcal{L}\left(  E;H\right)  \setminus\Pi
_{1}\left(  E;H\right)  \right)  \cup\{0\}.
\]

\end{proof}

\textbf{Acknowledgment.} Part of this paper was done while the second named author was visiting the third named author at Universitat de  Val\`encia. He thanks for the warm hospitality and nice research atmosphere during the visit.

\end{document}